\pdfoutput=1
\documentclass{article}
\include{header}

\usepackage[nonatbib, preprint]{neurips_2019}

\usepackage[utf8]{inputenc} 
\usepackage[T1]{fontenc}    
\usepackage{hyperref}       
\usepackage{url}            
\usepackage{booktabs}       
\usepackage{amsfonts}       
\usepackage{nicefrac}       
\usepackage{microtype}      
\usepackage[ruled,vlined]{algorithm2e}

\usepackage{wrapfig}
\usepackage{multirow}

\usepackage{graphicx, amsmath,amsthm, amssymb}

\usepackage{float}
\usepackage{subfigure}

\usepackage{courier}
\usepackage{color}
\definecolor{colorpink}{RGB}{251,53,155}
\definecolor{colorblue}{RGB}{0,148,200}
\definecolor{colorgreen}{RGB}{0,150,0}
\usepackage{alphalph}

\usepackage[font=small, justification=centering]{caption}
\DeclareCaptionFont{tiny}{\tiny}

\def\W{\mathbf{W}}
\def\I{\mathbf{I}}

\def\P{\mathbf{P}}

\def\X{\mathbf{X}}
\def\Y{\mathbf{Y}}

\def\A{\mathbf{A}}
\def\B{\mathbf{B}}

\def\St{\text{s.t.}}

\def\R{\mathbf{R}}

\def\0{\mathbf{0}}
\def\1{\mathbf{1}}

\def\b{\mathbf{b}}
\def\c{\mathbf{c}}

\def\x{\mathbf{x}}
\def\\x{\overline{\mathbf{x}}}
\def\y{\mathbf{y}}
\def\\y{\overline{\mathbf{y}}}
\def\z{\mathbf{z}}
\def\sgn{\text{sgn}}

\newtheorem{theorem}{Theorem}
\newtheorem{definition}{Definition}
\newtheorem{lemma}{Lemma}

\newtheorem{cor}{Corollary}

\title{Fast Large-Scale Discrete Optimization Based on Principal Coordinate Descent}

\author{
  Huan Xiong
   \And
   Mengyang Yu
   \And
   Li Liu
  \And
   Fan Zhu
  \And
   Fumin Shen
  \And
   Ling Shao
}

\begin{document}

\maketitle

\begin{abstract}
Binary optimization, a representative subclass of discrete optimization, plays an important role in mathematical optimization and has various applications in computer vision and machine learning. Usually, binary optimization problems are NP-hard and difficult to solve due to the binary constraints, especially when the number of variables is very large. Existing methods often suffer from high computational costs or large accumulated quantization errors, or are only designed for specific tasks. In this paper, we propose a fast algorithm to find effective approximate solutions for general binary optimization problems. The proposed algorithm iteratively solves minimization problems related to the linear surrogates of loss functions, which leads to the updating of some binary variables most impacting the value of loss functions in each step. Our method supports a wide class of empirical objective functions with/without restrictions on the numbers of $1$s and $-1$s in the binary variables. Furthermore, the theoretical convergence of our algorithm is proven, and the explicit convergence rates are derived, for objective functions with Lipschitz continuous gradients, which are commonly adopted in practice. Extensive experiments on several binary optimization tasks and large-scale datasets demonstrate the superiority of the proposed algorithm over several state-of-the-art methods in terms of both effectiveness and efficiency.
\end{abstract}

\section{Introduction}
Binary optimization problems are generally formulated as follows:
\begin{align}\label{Eq:Problem0}
  \min_{\mathbf{x}}~  f(\mathbf{x}), \,\,\,\,\,\,\,\,
  \St~~  \mathbf{x}\in    \{ \pm 1\}^{n}.
\end{align}
Problem \eqref{Eq:Problem0} appears naturally in several fields of computer vision and machine learning, including clustering \cite{wang2011information},  graph bisection \cite{wang2017large,yuan2016binary}, image denoising \cite{bi2014exact}, dense subgraph discovery \cite{ames2015guaranteed,balalau2015finding,yuan2016binary,yuan2013truncated}, multi-target tracking \cite{shi2013multi}, and community discovery \cite{he2016joint}. In many application scenarios, such as binary hashing  \cite{gui2018fast,liu2014discrete,shen2015supervised,shen2018unsupervised,wang2018survey,weiss2009spectral,xiong2021generalized}, Problem \eqref{Eq:Problem0} needs to be solved for millions of binary variables,  which makes the size $2^n$ of the feasible set very large (far larger than the number of atoms in the universe). Usually, it is difficult to find the optimal solution. Therefore, providing a fast algorithm to approximately solve Problem \eqref{Eq:Problem0} is very important in practice.

Furthermore, additional constraints on the numbers of $1$s and $-1$s in the binary variables $\x$ are adopted in many cases. For example, in binary hashing \cite{shen2015supervised,shen2018unsupervised} and graph bisection \cite{wang2017large,yuan2016binary}, the balance condition is often required, which means that the numbers of $1$s and $-1$s are equal to each other. On the other hand, dense subgraph discovery  \cite{ames2015guaranteed,balalau2015finding,yuan2016binary,yuan2013truncated} and information theoretic clustering \cite{wang2011information} require that the numbers of $1$s and $-1$s in $\x$ are some fixed integers.


To handle these previously mentioned constraints, in this paper, we focus on the following binary optimization problem:
\begin{align}\label{Eq:Problem1}
  \min_{\mathbf{x}}~  f(\mathbf{x}), \,\,\,\,\,\,\,\,
  \St~~  \mathbf{x}\in    \sl{\Omega_r},
\end{align}
where
$\x$ is a binary vector of length $n$, 
$f(\cdot)$ is a differentiable objective function (which may be nonconvex), $r\geq -1$ is a given integer, and the restriction $\Omega_r$ on $\x$ is defined as
\begin{align}
\Omega_{r}=
    \begin{cases}
    \{ \pm 1\}^{n},  \ \ \ \,\text{if}\,\,\, r=-1;
    \\
    \{ \x\in \{ \pm 1\}^{n}: \1^\top\x=2r-n\},\ \ \ \,\text{if}\,\,\, r\in\mathbb{N}_{\geq 0},
    \end{cases}
\end{align}
where $\mathbb{N}_{\geq 0}$ denote the set of nonnegative integers.
When $r=-1$, Problem \eqref{Eq:Problem1} is a binary optimization problem without further constraints. When~$r\in\mathbb{N}_{\geq 0}$, Problem \eqref{Eq:Problem1} becomes an optimization problem with the restriction that there are exactly $r$ $1$s in the binary vector $\x$. For instance, when $r=n/2$, the constraint $\1^\top\x=2r-n=0$ implies that the number of $1$s is equal to the number of $-1$s in $\x$.



In general, Problem \eqref{Eq:Problem1} is NP-hard due to the binary constraints \cite{johnson1979computers}. Many algorithms, such as continuous relaxation, equivalent optimization, signed gradient optimization and direct discrete optimization, have been proposed to solve it approximately (for details, please refer to Section \ref{sec:related}). However, they usually suffer from high computational costs or large accumulated quantization errors, or are only designed for specific tasks.

To overcome these difficulties, in this paper, we propose a novel and fast  optimization algorithm, termed \textit{Discrete Principal Coordinate Descent (DPCD)}, to approximately solve Problem \eqref{Eq:Problem1}. 
The time complexity for the binary optimization problem is relatively high when directly applying signed gradient methods (updating all variables at each time based on the gradients). On the contrary, the proposed DPCD focuses on the principal coordinates most impacting the value of the loss function. At each iteration, DPCD can adaptively decide the number of principal coordinates that need to be optimized, which can be regarded as analogous to the adaptive learning rate in the normal gradient descent methods. Different from other binary optimization algorithms in the literature, our DPCD method supports a large family of empirical objective functions with/without restrictions on the numbers of $1$s and $-1$s in the binary variables.
Furthermore, we prove theoretical convergence of DPCD for loss functions with Lipschitz continuous gradients, which cover almost every loss function in practice. Explicit convergence rates are also derived. Extensive experiments on two binary optimization tasks: dense subgraph discovery and  binary hashing, show the superiority of our method over state-of-the-art methods in terms of both efficiency and effectiveness.





\section{Related work} \label{sec:related}
 A very rich literature and a wide range of promising methods exist in binary optimization. We briefly review three classes of representative and related methods.

\paragraph{Continuous relaxation methods.} An intuitive method to approximately solve Problem \eqref{Eq:Problem0} is to relax the binary constraints to continuous variables, then threshold the continuous solutions to binary vectors. For instance, in the Linear Programming (LP) relaxation \cite{hsieh2015pu,komodakis2007approximate}, the binary constraint is substituted with the box constraint, i.e., $\x\in [-1,1]^n$, which can be approximately solved by continuous optimization methods such as the interior-point method \cite{mehrotra1992implementation}. On the other hand, the Semi-Definite Programming (SDP) relaxation \cite{wang2017large} replaces the binary constraint with some positive semi-definite matrix constraint. In Spectral relaxation \cite{lin2013general,olsson2007solving}, the binary constraint is relaxed to some $\ell_2$-ball, which is non-convex. One of the advantages of such continuous relaxation methods is that the relaxed problems can be approximately solved efficiently by existing continuous optimization solvers. However, the relaxation is usually too loose, and the thresholding often yields large quantization errors.

\paragraph{Equivalent optimization methods.} Unlike relaxation methods, equivalent optimization methods replace the binary constraint with some equivalent forms, which are much easier to handle. For example, motivated by linear and spectral relaxations, Wu and Ghanem \cite{wu2018lp} replaced the binary constraint with the intersection of the box $[-1,1]^n$ and the sphere $ \{ \x:\|\x \|_2^2 = n \} $, and then applied the Alternating Direction Method of Multipliers (ADMM) \cite{boyd2011distributed,li2015global,wang2015global} to solve the optimization problem iteratively. Other methods in this direction include the MPEC-ADM and MPEC-EPM methods (\cite{yuan2016binary, yuan2017exact}),  the $\ell_0$ norm reformulation \cite{lu2013sparse,yuan2016sparsity}, the $\ell_2$ box non-separable reformulation \cite{murray2010algorithm}, and the piecewise separable reformulation \cite{zhang2007binary}. Usually, these equivalent optimization methods guarantee the convergence to some stationary and feasible points, but the convergence speed is often too slow, resulting in high computational costs for  large-scale optimization problems.

 \paragraph{Signed gradient methods.} In the Signed Gradient Method (SGM) \cite{liu2014discrete}, a linear surrogate of the objective function $f(\x)$ is given at each iteration. Then, the minimization (actually a maximization problem was studied in the original paper \cite{liu2014discrete}, we state an equivalent form here) of this surrogate function gives the updating rule for Problem \eqref{Eq:Problem0} as: $ \x^{k+1}=-{\sgn}(\nabla f(\x^k)) $. The sequence obtained by this updating rule is guaranteed to converge if the objective function is concave. However, even for a very simple non-concave function, SGM may generate a divergent sequence and never converge (please refer to Lemma \ref{prop:example}). Furthermore, SGM cannot handle binary problems with restriction on the number of $1$s  since  the number of $1$s may change during each iteration.  A stochastic version of this method was given in Adaptive Discrete Minimization (ADM) \cite{liu2017discretely}, in which an adaptive ratio $\psi$ was selected at each iteration, then some random $\psi n$ entries of $\x$ were updated by $ \x^{k+1}_i=-{\sgn}(\nabla_i f(\x^k)) $.  Although ADM works well for certain loss functions, it fails when the value of the loss function depends largely on only a few variables, since the random selecting procedure may skip such important variables.

\paragraph{Discrete optimization methods.}
In the field of image hashing, many direct discrete optimization methods, such as DCC  \cite{shen2015supervised},  SADH \cite{shen2018unsupervised}, ARE \cite{hu2018hashing}, ITQ \cite{gong2013iterative} and FastHash \cite{lin2014fast}, have been proposed, which aim to optimize binary variables directly (SGM  can also be seen as a discrete optimization method).  For example, the Coordinate Descent (CD) method \cite{wright2015coordinate} is widely used for solving optimization problems with smooth and convex constraints. Motivated by this method, several  discrete cyclic coordinate descent (DCC) methods (e.g., RDCM \cite{luo2018robust}, FSDH \cite{gui2018fast}, and SDH  \cite{shen2015supervised}) have been proposed to handle the binary constraint directly. The main idea is that, at each iteration, a subproblem with most entries of the binary variables fixed is considered, and the loss function is minimized with respect to the remaining entries.
Although such methods can work well for specific loss functions, most of them are usually difficult to be extended to handle general binary optimization problems. Furthermore, they often suffer from expensive computational costs.


\section{Proposed algorithm}
In this section, we present in detail the DPCD algorithm for solving Problem \eqref{Eq:Problem1}, which is a general binary optimization problem with/without  restrictions on the numbers of $1$s and $-1$s. We also provide a theoretical convergence analysis of the proposed method.

\subsection{Notation and preliminaries}\label{sec:notations}
We first introduce some notation and preliminaries. A vector is represented by some lowercase bold character, while a matrix is represented by some uppercase bold character. Let $\x_i$ and $\A_{ij}$ denote the $i$-th and $(i,j)$-th entries of a vector $\x$ and a matrix $\A$, respectively. The transpose of a matrix $\A$ is represented by $\A^\top$.
We use $\langle \cdot,\cdot \rangle$ to denote the Euclidean inner product. Let $\|\A\|=\sqrt{\sum_{ij}\A_{ij}^2}$ and $\|\A\|_1=\sum_{ij}|\A_{ij}|$ be the Frobenius norm and $1$-norm of a matrix $\A$, respectively. The gradient of a differentiable function $f(\x)$ is denoted by $\nabla f(\x)=(\nabla_1 f(\x),\nabla_2  f(\x),\cdots,\nabla_n f(\x))$.
Let ${\sgn}(\x)=({\sgn}(\x_1),{\sgn}(\x_2),\cdots,{\sgn}(\x_n))$ denote the  element-wise sign function where ${\sgn}(\x_i)=1$ for $\x_i \geq 0$ and $-1$ otherwise. The Hamming distance between two binary vectors $\y$ and $\z$ of equal length is defined by $d_{H}(\y,\z)$, which is the number of positions at which the corresponding entries are different. For a set $S$, let $\#S$ denote the number of elements in $S$.

\subsection{Main algorithm}
The proposed DPCD algorithm runs iteratively between principal coordinate update and neighborhood search.

\paragraph{Principal coordinate update.} The basic idea is that, in the $k$-th iteration, we change the sign of some adaptively chosen entries of the binary vector $\x^k$, such that the value of the loss function should decrease steeply after each change. To achieve this goal, we focus on $L$-principal coordinates (see the following definition), which have major influences on the value change of the loss function.
\begin{definition}
Let $f(\x)$ be a differentiable function and
$L>0$ be a positive constant. A coordinate index $i$ is called an {\sl $L$-principal coordinate} of $\x \in \{-1, 1\}^n$ if the product ~$ \x_i \cdot \nabla_i f(\x) \geq L$.
\end{definition}
One motivation of our method is SGM \cite{liu2014discrete}. In the $k$-th iteration of SGM, the linear surrogate of the objective function $f(\x)$ is given as:
\begin{align}
\hat{f}^k(\x)=f(\x^k)+\langle \nabla f(\x^k), \x-\x^k  \rangle.
\label{eq:surrogate}
\end{align}
Then $\x^{k+1}$ is obtained by minimizing this surrogate function:
\begin{align}
    \x^{k+1}=\arg\min_{\x\in \{\pm 1\}^n} \hat{f}^k(\x)=-{\sgn}(\nabla f(\x^k)).  \label{eq:updatexk}
\end{align}
The sequence obtained by Eq. \eqref{eq:updatexk} is guaranteed to converge if $f(\x)$ is concave. However, from Lemma \ref{prop:example} in Subsection \ref{sec:compare} we know, for non-concave functions, SGM \cite{liu2014discrete} may generate divergent sequences and never converge, since changing too many entries of $\x$ at a time may increase the value of the loss function.
To overcome this difficulty, the proposed DPCD method only changes the signs of entries whose absolute values of directional derivatives are large (entries with principal coordinates).  This yields convergence for a wide class of loss functions (please refer to Lemma \ref{prop:example} and Theorem \ref{th:converhgence}). Also, in the proposed algorithm, the constraint $\Omega_r$ is always satisfied after each iteration.

To be more specific, given $\x^k$ at the $k$-th iteration, we first calculate the gradient $\nabla f(\x^k)=(\nabla_1 f(\x^k), \nabla_2 f(\x^k),\cdots,\nabla_n f(\x^k))$ for the differentiable loss function $f(\x)$. Next, we derive some proper thresholds $L_1, L_2$ based on $\nabla f(\x^k)$. 
When $\nabla f$ is $L_0$-Lipschitz continuous on $[-1,1]^n$, where $L_0$ is easy to calculate, we simply set
\begin{align}\label{eq:derive_L0}
L_1= L_2=L_0+\epsilon
\end{align}  for some sufficiently small positive constant  $\epsilon>0$, i.e., we consider $(L_0+\epsilon)$-principal coordinates. For example, in Lemma \ref{prop:example}, it is easy to see that  $L_0=1$, then we take $L_1=L_2=1+\epsilon$ for some small $\epsilon>0$ in the algorithm. When $L_0$ does not exist or is difficult to compute, we let $L_1$ and $L_2$ be the averages of the absolute values of the positive and negative entries in the gradient $\nabla f(\x^k)$, respectively, i.e.,
\begin{align}\label{eq:derive_L}
L_1=\frac{1}{n_1}\sum_{\nabla_i f(\x^k)>0} \nabla_i f(\x^k) \,\, \text{and} \,\, L_2=-\frac{1}{n_2}\sum_{\nabla_i f(\x^k)<0} \nabla_i f(\x^k),
\end{align}
where $n_1$ and $n_2$ are the numbers of positive and negative entries in $\nabla f(\x^k)$, respectively. With the given thresholds $L_1$ and $L_2$, we set $S^{k+}=\{1\leq i \leq n: \nabla_i f(\x^k)>\alpha_1 L_1,\, \x^k_i=1\}$ and $S^{k-}=\{1\leq i \leq n: \nabla_i f(\x^k)<-\alpha_2 L_2,\, \x^k_i=-1\}$ to be the sets of $L_1$-principal coordinates with positive partial derivatives and $L_2$-principal coordinate with negative partial derivatives, respectively, where $\alpha_1$ and $\alpha_2$ are some parameters in $[0.1,10]$ which will be learned depending on the tasks and datasets.
If the restriction condition is $\x\in \Omega_{-1}$, we update $\x^{k+1}_i$ by solving $\min_{\x\in \{\pm 1\}^n} \hat{f}^k(\x)$ in Eq. \eqref{eq:updatexk} with respect to $i \in S^{k+} \cup S^{k-}$ (other entries of $\x$ are fixed) and derive:
\begin{align}
    \x^{k+1}_{i}=-\sgn(\nabla_{i} f(\x^k))=-\x^{k}_{i}.   \label{eq:updateBk}
\end{align}
In other words, we change the sign of $\x^k_i$ for $\x^k_i=1$ with $\alpha_1 L_1$-principal coordinates, and for $\x^k_i=-1$ with $\alpha_2 L_2$-principal coordinates.
If the number of $1$s in $\x$ is required to be fixed (i.e., the restriction condition is $\x\in \Omega_{r}$ for some $r\in \mathbb{N}$), we update Eq. \eqref{eq:updateBk}  with respect to the $m$ largest absolute values in $\{|\nabla_i f(\x^k)|: i\in S^{k+}\}$ and $\{|\nabla_j f(\x^k)|: j\in S^{k-}\}$, respectively, where $m=\min\{\#S^{k+},\, \#S^{k-}\}$ (such procedure guarantees that $\x^k\in \Omega_{r}$ implies $\x^{k+1}\in \Omega_{r}$).
When the complexity of gradient calculations is low, the updating is very fast.
A complexity analysis of the proposed DPCD is given in Subsection \ref{sec:hashing}, which shows that the algorithm complexity of DPCD for supervised discrete hashing is linear and thus the algorithm runs very fast for large-scale datasets.


\paragraph{Neighborhood search.} We add an optional heuristic neighborhood search after the principal coordinate update to avoid saddle points. In practice, we run one neighborhood search after $T$ principal coordinate updates, where $T$ is between 10 to 20.
First, we define the concept of \emph{$m$-neighbors} for a point $\x\in \Omega_r$ where $m\in\mathbb{N}$ is some small positive integers. When $r =-1$, the set of $m$-neighbors for $\x\in \Omega_{-1}$ is denoted by $N_{-1}(\x):=\{  \y \in \{\pm 1\}^n: 0<d_{H}(\y,\x) \leq m \}$, which is the set of points with a Hamming distance at most $m$ from $\x$. When $r \geq 0$, the set of $m$-neighbors for $\x\in \Omega_{r}$ is denoted by $N(\x):=\{  \y\in \{\pm 1\}^n: 0<d_{H}(\y,\x) \leq 2m,\,\, \sum_{i=1}^n \y_i = \sum_{i=1}^n \x_i \}$, which is the set of points obtained by interchanging at most $m$ pairs of entries $1$ and $-1$ in $\x$. For instance, when $m=1$, we have $N_{-1}((1,-1,1,1)):=\{  (-1,-1,1,1),(1,1,1,1),(1,-1,-1,1),(1,-1,1,-1)  \}$ and $N((1,-1,1,1)):=\{  (-1,1,1,1),(1,1,-1,1),(1,1,1,-1)\}$.

In a neighbor search for some $\x^*\in \Omega_{-1}$ (or $\Omega_{r}$ with $r\in\mathbb{N}_{\geq 0}$), the aim is to find the point $\y\in N_{-1}(\x^*)\cup \{\x^*\}$ (or~$N(\x^*)\cup \{\x^*\}$) with the minimal function value $f(\y)-f(\x^*)$ (or equivalently, $f(\y)$). In practice, we can sample from $N_{-1}(\x^*)$ (or $N(\x^*)$) instead of iterating over all points, if $n$ is large or calculating $f(\y)-f(\x^*)$ is slow.
This neighborhood search step is helpful for finding a local minimum point.
The proposed algorithm is summarized in Algorithm \ref{Alg:1}.

\begin{algorithm}[H]\caption{Discrete Principal Coordinate Descent (DPCD) }\label{Alg:1}
\small{
\KwIn{Loss function $f(\x)$, code length $n$, the restriction $\Omega_r$ where $r=-1 \,\text{or}\, r\in \mathbb{N}_{\geq 0}$, parameters $\alpha_1$, $\alpha_2$.}
\KwOut{Binary codes $\x^*$.}

Initialize $\x^*$ by the sign of some random vector according to $\Omega_r$;~ $\x^1=\x^*$ and $k=1$;

\While{not converge or not reach maximum iterations}{
Calculate the gradient $\nabla f(\x^k)=(\nabla_1 f(\x^k), \nabla_2 f(\x^k),\cdots,\nabla_n f(\x^k))$;

Derive proper thresholds $L_1, L_2$ by Eq. \eqref{eq:derive_L0} or Eq. \eqref{eq:derive_L};

Build sets $S^{k+}=\{i: \nabla_i f(\x^k)>\alpha_1 L_1,\, \x^k_i=1\}$ and $S^{k-}=\{i: \nabla_i f(\x^k)<-\alpha_2 L_2,\, \x^k_i=-1\}$;

    \eIf{the restriction condition is $\x\in \Omega_{-1}$ (i.e., $r=-1$)}
    {
        Update $\x^{k+1}_{i}=-\sgn(\nabla_{i} f(\x^k))=-\x^{k}_{i}$ for $i \in S^{k+} \cup S^{k-}$;
    }
    {
        Sort $\{|\nabla_i f(\x^k)|: i\in S^{k+}\}$ and $\{|\nabla_j f(\x^k)|: j\in S^{k-}\}$ in descending order as $|\nabla_{i_1} f(\x^k)| \geq |\nabla_{i_2} f(\x^k) |  \geq |\nabla_{i_3} f(\x^k)| \geq \cdots$, and  $|\nabla_{j_1} f(\x^k)|\geq |\nabla_{j_2} f(\x^k)|  \geq |\nabla_{j_3} f(\x^k)| \geq\cdots$, respectively;

        Update $\x^{k+1}_{i_l}=-\x^{k}_{i_l}$ and $\x^{k+1}_{j_l}=-\x^{k}_{j_l}$ for $1\leq l \leq \min\{\#S^{k+},\, \#S^{k-}\}$;
    }

(Optional) Neighborhood search for $\x^{k+1}$;

$k=k+1$;
}
\textbf{Return} $\x^*=\x^{k+1}$.
}
\end{algorithm}

\subsection{Convergence comparison: DPCD vs. SGM} \label{sec:compare}

One of the differences between our DPCD and SGM \cite{liu2014discrete} is the choice of $\x_i^k$ that should be updated by Eq. \eqref{eq:updateBk} at the $k$-th iteration. SGM updates Eq. \eqref{eq:updateBk} for each $i$, while the proposed DPCD only changes the sign of $\x_i^k$ when the coordinate indexes is $L$-principal for some $L$. This difference is crucial to the convergences of the algorithms. More specifically, SGM can only guarantee convergence for the minimization of concave functions, while DPCD converges in finite steps for any functions with Lipschitz continuous gradients. We show the superiority of the proposed DPCD by the following example. The case $n=2$ is presented in Figure \ref{fig:lemma1}.

\begin{wrapfigure}{r}{0.37\textwidth}
\vspace{-5ex}
\includegraphics[width=0.37\textwidth]{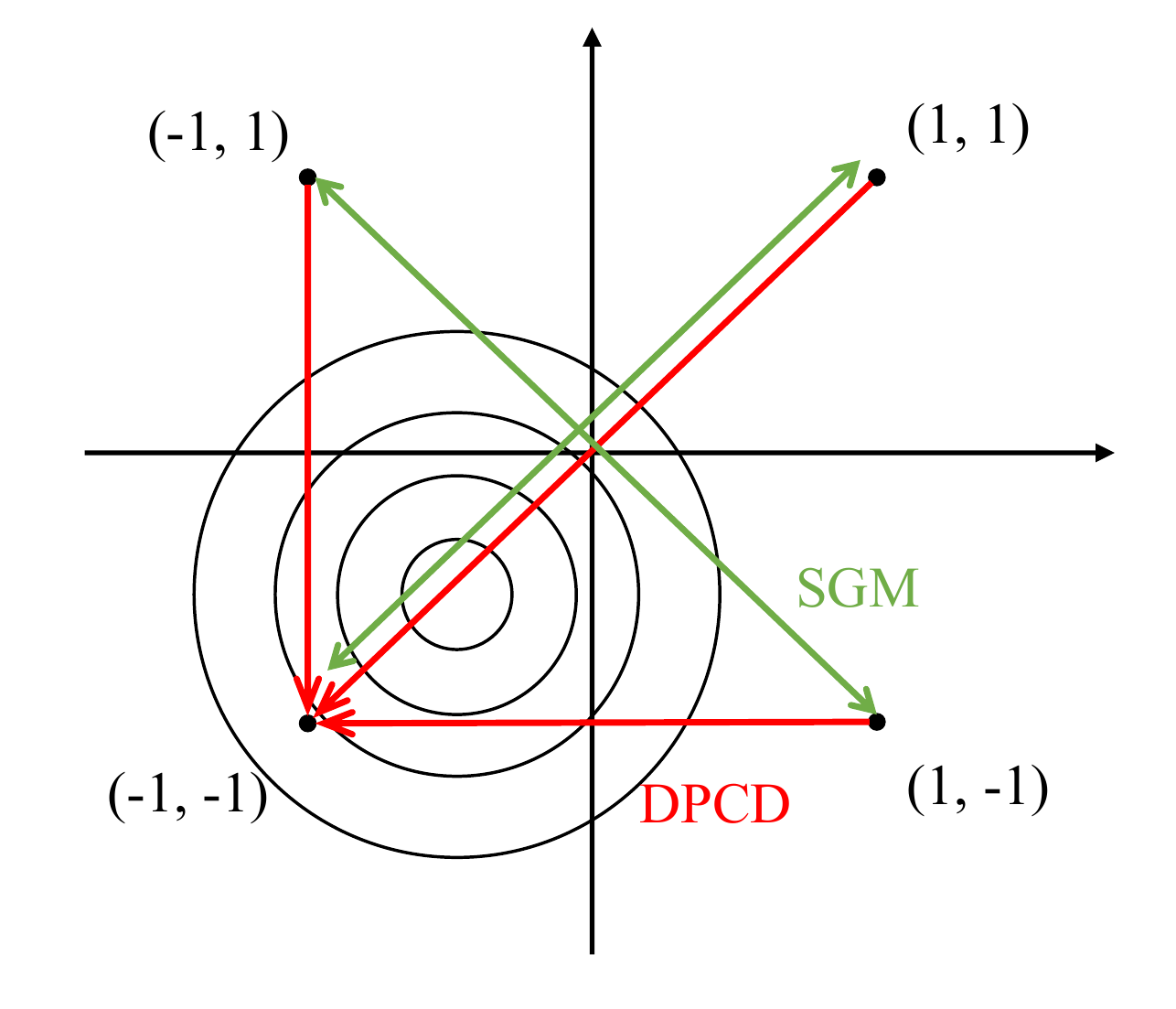}
\vspace{-3ex}
\caption{The optimization routes of DPCD and SGM on a 2-d example. The black circles are the contour lines of the objective function. Our method DPCD can converge to the optimum from any initialization points in only one step, while~SGM always oscillates between two points. In fact, each $\beta_i$ can be in the open~interval (-1,1), and the optimum is~decided by the signs of all $\beta_i$.  }\label{fig:lemma1}
\vspace{-4ex}
\end{wrapfigure}

\begin{lemma}\label{prop:example}
Consider the problem
\begin{equation*}
\min_{\x\in   \{ \pm 1\}^n} f(\x) := \min_{(\x_1,\x_2,\cdots,\x_n)\in   \{ \pm 1\}^n} ~  \frac{1}{2}\sum_{i=1}^n (\x_i+{\beta_i})^2,
\end{equation*}
where $0<\beta_i<1$. Let $\beta=\min_i \beta_i$.
It holds that:
(1)
SGM generates a divergent sequence for any initial point;
(2)
With parameters $\alpha_1=\alpha_2=1$ and $0<\epsilon<\beta$, the proposed DPCD method always converges to the optimal solution.
\end{lemma}

\begin{proof}
First, we apply the method SGM to this problem. The gradient of $f(\x)$ is straightforward:
$\nabla f(\x) = (\x_1+{\beta_1},\x_2+{\beta_2},\cdots,\x_n+{\beta_n})$.
Since $0<\beta_i<1$, we have $\x=\sgn(\nabla f(\x))$ for any $\x\in   \{ \pm 1\}^n$. Then, by the updating rule fo SGM, $\x^{k+1}=-\sgn(\nabla f(\x^k))=-\x^k$. Therefore, starting from any point $\x\in   \{ \pm 1\}^n$, SGM always generates a divergent sequence $\x,-\x,\x,-\x,\x,-\x,\cdots$.

On the other hand,  since $\|\nabla f(\y)
-\nabla f(\z)\| = \|\y-\z\|$ for any $\y,\z\in   \{ \pm 1\}^n$, $f(\x)$ is a function with $1$-Lipschitz continuous gradient. Therefore, in DPCD, we set $1<L_1=L_2=1+\epsilon<1+\beta$. Then by definition,
\begin{align*}
  S^{k+}&=\{1\leq i \leq n: \nabla_i f(\x^k)>1+\epsilon,\, \x^k_i=1\}
  \\&=
  \{1\leq i \leq n: \x^k_i+\beta_i>1+\epsilon,\, \x^k_i=1\}
  \\&=
  \{1\leq i \leq n: \, \x^k_i=1\},
\end{align*}
 and
 \begin{align*}
 S^{k-}&=\{1\leq i \leq n: \nabla_i f(\x^k)<-1-\epsilon,\, \x^k_i=-1\}
 \\&=
 \{1\leq i \leq n: \x^k_i+\beta_i<-1-\epsilon,\, \x^k_i=-1\}
 \\&=
 \emptyset.
 \end{align*}
This implies that
\begin{align}
\x^{k+1}_{i}=
    \begin{cases}
    -{\sgn}(\nabla_{i} f(\x^k))=-\x^{k}_{i}=-1, \,\text{if}\,\, i \in S^{k+}  = S^{k+} \cup S^{k-}
    \\
    \x^{k}_{i}=-1, \,\text{if}\,\, i \notin S^{k+}.
    \end{cases}
\end{align}
 Then, we have $\x^{k+1}_{i}=-1$ for any $k\geq 2$ and $i$, and thus $\x^{2}=\x^{3}=\cdots=\x^{k}=(-1,-1,\cdots,-1)$ for $k\geq 2$. It is easy to check that $(-1,-1,\cdots,-1)$ is indeed the optimal point for the problem in Lemma $1$. Therefore, the proposed DPCD converges to the optimal solution in only one updating step for any initial point.
\end{proof}

\subsection{Theoretical convergence results} \label{sec:Convergence}
For simplicity, we ignore the neighborhood search part in the convergence analysis. Actually, the neighborhood search does not have any influence on the convergence since it always generates some binary vectors with non-increasing values of the loss function. Thus, we can only focus on the principal coordinate update part of the proposed DPCD method.

We derive the following convergence results.  When we say the algorithm converges in $T$ steps, we mean that the binary vector $\x^{T+1}$ obtained in $(T+1)$-th iteration equals $\x^{T}$ in $T$-th iteration (then the algorithm can stop here). Moreover, it is easy to see that our algorithm can converge to some local optimum with the help of neighborhood search (without neighborhood search, it only converges to some fixed binary vectors).

\begin{theorem}\label{th:converhgence}
Let $f :  \mathbb{R}^n \rightarrow \mathbb{R}$ be a differentiable function such that
$\nabla f$ is $L_0$-Lipschitz continuous on $[-1,1]^n$. Setting the thresholds $L_1=L_2=L_0+\epsilon$ where $\epsilon>0$, and ignoring the neighborhood search, Algorithm \ref{Alg:1} always converges in  $\frac{f_{max}-f_{min}}{2\epsilon}$ steps at most, where $f_{max}:=\max_{\mathbf{x}\in   \{ \pm 1\}^{n} } f(\x)$ and $f_{min}:=\min_{\mathbf{x}\in   \{ \pm 1\}^{n} } f(\x)$.
\end{theorem}

\begin{proof}
Let $L=L_1=L_2=L_0+\epsilon$.
Since $\nabla f$ is $L_0$-Lipschitz, we have for any $\x,\y\in [-1,1]^n$,
\[
\|\nabla f(\x) - \nabla f(\y)\| \le L_0 \| \x-\y\|.
\]
Then, by Cauchy-Schwarz inequality,
\[
(\nabla f(\x) - \nabla f(\y)^\top(\x-\y)) \leq \|\nabla f(\x) - \nabla f(\y)\|\cdot \| \x-\y \| \leq  L_0 \rVert \x-\y\rVert^2.
\]
By the gradient monotonicity equivalence of convexity \cite[Page 40]{borwein2010convex}, this yields that $g(x) = \frac{L_0}{2}\|\x\|^2 - f(\x)$ is a convex function on $[-1,1]^n$. Therefore, due to the first-order equivalence of convexity \cite[Page 69]{boyd2004convex},
\[
\frac{L_0}{2}\|\y\|^2 - f(\y) \geq \frac{L_0}{2}\|\x\|^2 - f(\x) + (L_0\x-\nabla f(\x))^\top (\y-\x),
\]
which implies that
\[
f(\y)\le f(\x)+\nabla f(\x)^\top(\y-\x)+\frac{L_0}{2}\lVert \y-\x\rVert^2,~\forall \x,\y  \in [-1,1]^n.
\]
Let $\x=\x^{k}$ and $\y=\x^{k+1}$ in the above inequality, where $\x^{k+1}$ is determined by $\x^{k}$ and (8) in the  paper. We have
\begin{align}\label{eq:L-Lip1}
   f(\x^{k+1})\le f(\x^{k})+\nabla f(\x^{k})^\top(\x^{k+1}-\x^{k})+\frac{L_0}{2}\lVert \x^{k+1}-\x^{k}\rVert^2.
\end{align}

If $\x^{k+1} = \x^{k}$, we obtain $\x^{j+1} = \x^{j}$ for each $j\geq k$ due to the updating rule.

If $\x^{k+1}_i \neq  \x^{k}_i$ for some $1\leq i \leq n$, by the updating rule we have, $\x^{k+1}_i = -{\sgn}(\nabla_{i} f(\x^k))= - \x^{k}_i$ and $\|\nabla_{i} f(\x^k)\| > L$.
 Then,  \[\nabla_{i} f(\x^k) (\x^{k+1}_i  - \x^{k}_i) = -2\,\nabla_{i} f(\x^k)\,  {\sgn}(\nabla_{i} f(\x^k))= -2\, \|\nabla_{i} f(\x^k)\| <-2L\]
and $\frac{L_0}{2}\lVert \x^{k+1}_i-\x^{k}_i\rVert^2=\frac{L_0}{2}\cdot 4=2L_0.$ Therefore,
\begin{align}\label{eq:L-Lip2} \nabla_{i} f(\x^{k})^\top(\x^{k+1}_{i}-\x^{k}_{i})+\frac{L_0}{2}\lVert \x^{k+1}_{i}-\x^{k}_{i}\rVert^2<-2(L-L_0).
\end{align}
If $\x^{k+1} \neq \x^{k}$, we know $\{j:\, \x^{k+1}_j \neq  \x^{k}_j\}$ is not empty. Then, by Eq. \eqref{eq:L-Lip1} and Eq. \eqref{eq:L-Lip2} we have,
\begin{align}\label{eq:L-Lip3}  \nonumber
    f(\x^{k+1}) - f(\x^{k})
    \leq &
    \nabla f(\x^{k})^\top(\x^{k+1}-\x^{k})+\frac{L_0}{2}\lVert \x^{k+1}-\x^{k}\rVert^2
    \\ = &\sum_{j=1}^n  \left(\nabla_{j}    \nonumber f(\x^{k})(\x^{k+1}_{j}-\x^{k}_{j})+\frac{L_0}{2}\lVert \x^{k+1}_{j}-\x^{k}_{j}\rVert^2\right)
    \\ = & \sum_{j:\,\x^{k+1}_j =  \x^{k}_j}  \left(\nabla_{j}   \nonumber  f(\x^{k})(\x^{k+1}_{j}-\x^{k}_{j})+\frac{L_0}{2}\lVert \x^{k+1}_{j}-\x^{k}_{j}\rVert^2\right)
    \\& + \sum_{j:\,\x^{k+1}_j \neq  \x^{k}_j}  \left(\nabla_{j}   \nonumber  f(\x^{k})(\x^{k+1}_{j}-\x^{k}_{j})+\frac{L_0}{2}\lVert \x^{k+1}_{j}-\x^{k}_{j}\rVert^2\right)
    \\ = & 0 + \sum_{j:\,\x^{k+1}_j \neq  \x^{k}_j}  \left(\nabla_{j}  \nonumber f(\x^{k})^\top(\x^{k+1}_{j}-\x^{k}_{j})+\frac{L_0}{2}\lVert \x^{k+1}_{j}-\x^{k}_{j}\rVert^2\right)
    \\< & -2(L-L_0)  \nonumber
    \\= & -2\epsilon
    .
\end{align}
This means that, in each updating step, the function value decreases at least $2\epsilon$, which completes the proof since the feasible set of Problem (2) in the main paper is finite.
\end{proof}

The above theorem is very versatile since most loss functions in practice are $L_0$-Lipschitz continuous on $[-1,1]^n$ for some $L_0 \in\mathbb{R}^+$. Furthermore, since $n$ is finite,  $ \{ \pm 1\}^{n} $ is a finite set, thus the above $f_{max}$ and  $f_{min}$ always exist and are finite. When $f(\x)$ is quadratic, we obtain the following direct corollary.

\begin{cor}\label{th:converhgence_cor}
Let $f(\x)=\x^\top \A \x + \c^\top\x+d$ be some quadratic function where $\A\in\mathbb{R}^{n\times n}$, $\x,\c\in \mathbb{R}^n$, and $d\in \mathbb{R}$, in Theorem \ref{th:converhgence}. Then Algorithm \ref{Alg:1} always converges in $\frac{\|\A\|_1+\|\c\|_1}{\epsilon}$ steps at most.
\end{cor}

\begin{proof}
Since $\x\in\{ \pm 1\}^{n}$, we have $-\|\A\|_1-\|\c\|_1+d\leq f(\x)\leq \|\A\|_1+\|\c\|_1+d$. Then $f_{max}-f_{min}\leq 2(\|\A\|_1+\|\c\|_1)$. By Theorem $1$ we complete the proof.
\end{proof}

\section{Experiments}\label{sec:experiments}
In this section, we compare the proposed DPCD algorithm with several state-of-the-art methods on two binary optimization tasks:  dense subgraph discovery and binary hashing.
All codes are implemented in MATLAB using a workstation with an Intel 8-core 2.6GHz CPU and 32GB RAM.


\subsection{Dense subgraph discovery}

\begin{table}
\caption{Statistics for the graphs used in the dense subgraph discovery experiments.}
\center
\resizebox{0.65\linewidth}{!}{
\begin{tabular}{cccc}
  \hline
 Graph & \# Nodes & \# Arcs & \# Arcs/\# Nodes \\
  \hline
 uk-2007-05 &100000&3050615& 30.506 \\
 dblp-2010 & 326186 & 1615400 & 4.952 \\
 eswiki-2013    &972933	& 23041488& 23.683  \\
 hollywood-2009 &1139905	& 113891327& 99.913\\
\hline
\end{tabular}
}
\label{table:graph}
\end{table}

\paragraph{Optimization problem.}
Dense subgraph discovery \cite{ravi1994heuristic,rozenshtein2018finding,yuan2016binary,yuan2013truncated} has many applications in graph mining, such as real-time story identification \cite{angel2012dense},  finding correlated genes \cite{zhang2005general} and graph visualization \cite{alvarez2006large}.  Let $\sl{G}$ be a given undirected weighted graph with $n$ nodes, and $k$ be a given positive integer such that $1\leq k\leq n$. The aim is to find the maximum density subgraph (the subgraph with the maximal sum of edge weights) with cardinality $k$. Let $\W\in\mathbb{R}^{n\times n}$ be the symmetric adjacency matrix of the graph~$G$, where $\W_{ij}$ denotes the weight of the edge between vertices $i$ and $j$. Then the problem can be formulated as the following optimization problem
\begin{equation}\label{eq:Problem dense2}
\max_{\x \in \{0,1\}^n} \,\, \x^{\top} \W \x, \quad\,\,\, \St ~~ \x^{\top} \mathbf{1} = k.
\end{equation}

Note that, in this case, the variables are $0$ or $1$ instead of $-1$ or $1$.
In order to translate the problem to the form in Problem \eqref{Eq:Problem1}, the substitution $\x=\frac{1}{2}(\y+\1)$ is adopted.
Therefore, Problem \eqref{eq:Problem dense2} is equivalent to:
\begin{equation}\label{eq:Problem dense3}
\min_{\y \in \{-1,1\}^n} \,\, -\y^{\top} \W \y- 2\y^\top\W\1-\1^\top\W\1, ~~~ \St ~~ \y^{\top} \mathbf{1} = 2k-n.
\end{equation}
In this way, the above problem can be approximately solved using our Algorithm \ref{Alg:1}.

\captionsetup[subfigure]{labelformat=empty}


\begin{figure*} [htbp!]

\centering

\subfigure{
\begin{minipage}[htbp!]{0.13\textwidth}
\centering
\vspace*{-0.1cm}
\includegraphics[width=2.15cm]{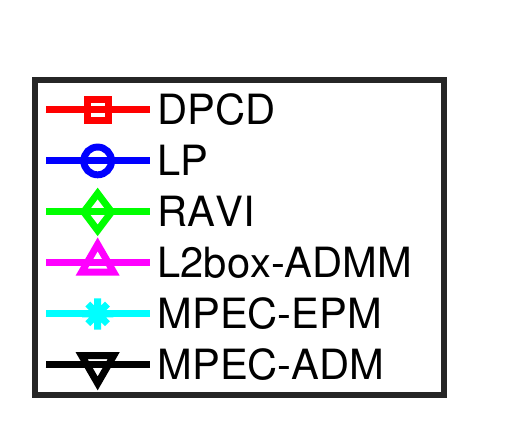}
\end{minipage}}
\addtocounter{subfigure}{-2}
\subfigure[(a)~ uk-2007-05]{
\begin{minipage}[htbp!]{0.195 \textwidth}
\centering
\vspace*{0.0cm}
\includegraphics[width=3.0cm]{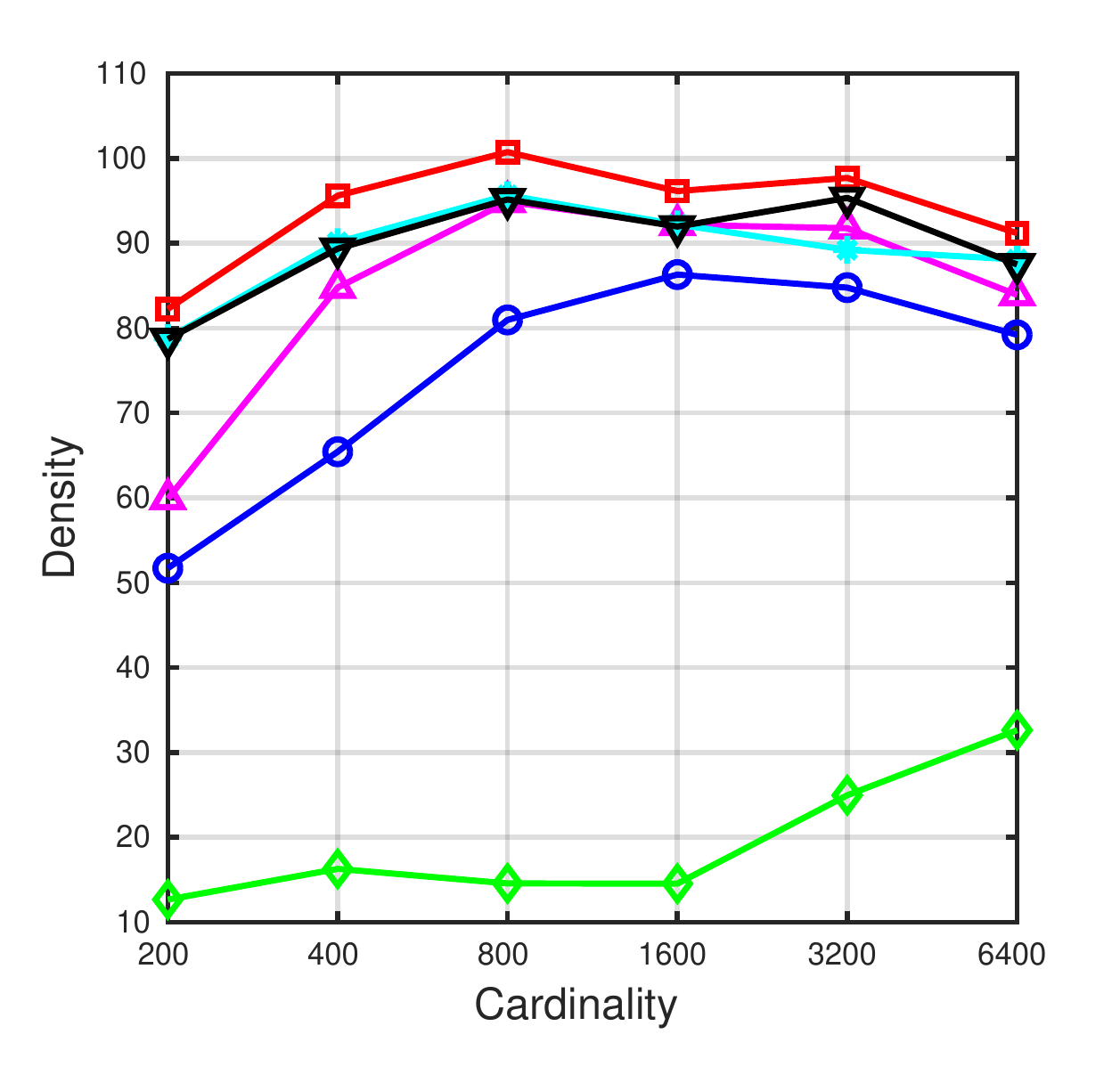}
\end{minipage}}
\addtocounter{subfigure}{-1}
\subfigure[(b)~ dblp-2010]{
\begin{minipage}[htbp!]{0.195 \textwidth}
\centering
\vspace*{0.0cm}
\includegraphics[width=2.93cm]{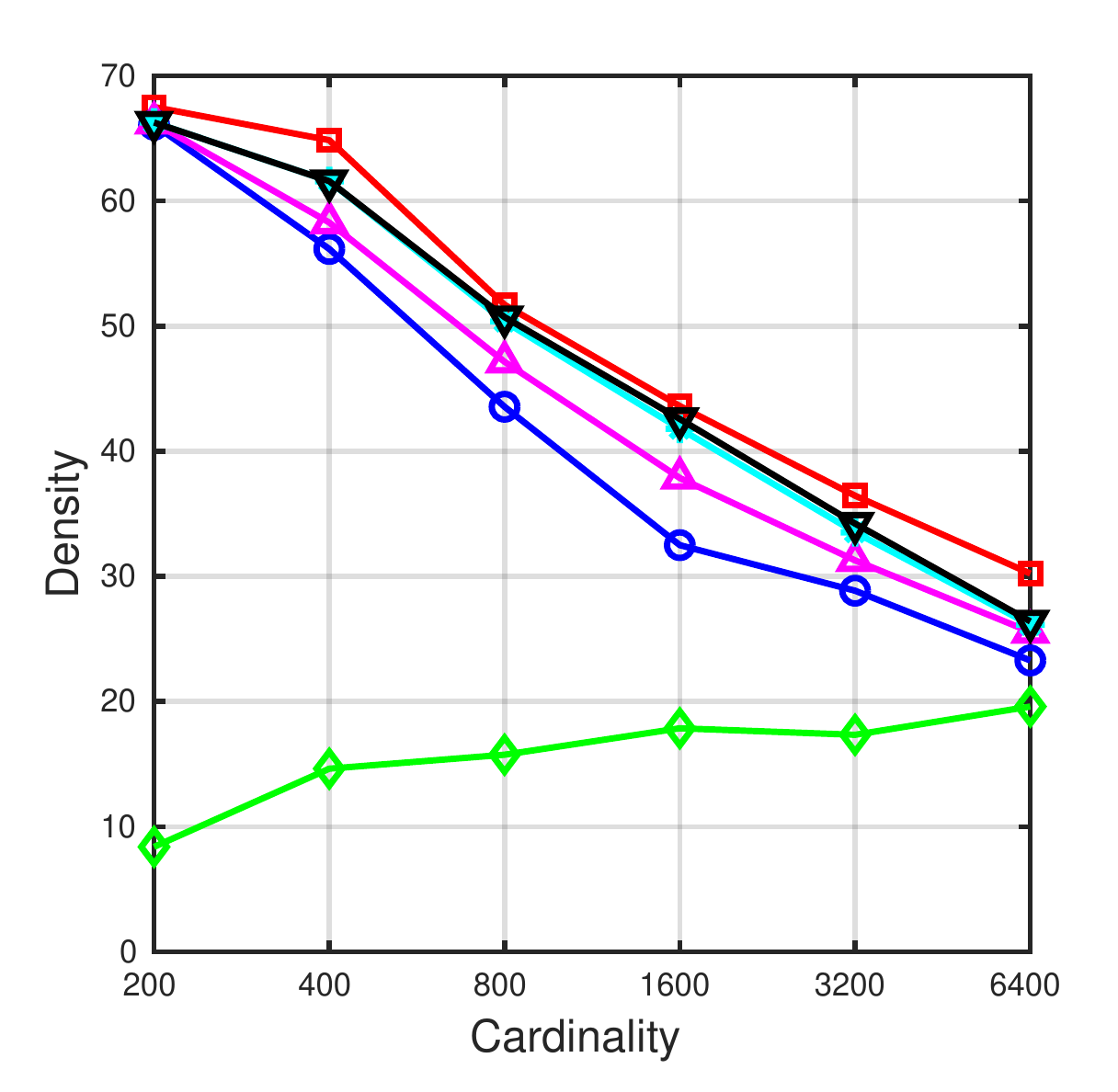}
\end{minipage}}
\addtocounter{subfigure}{-1}
\subfigure[~~(c)~ eswiki-2013]{
\begin{minipage}[htbp!]{0.195 \textwidth}
\centering
\vspace*{0.05cm}
\includegraphics[width=3.0cm]{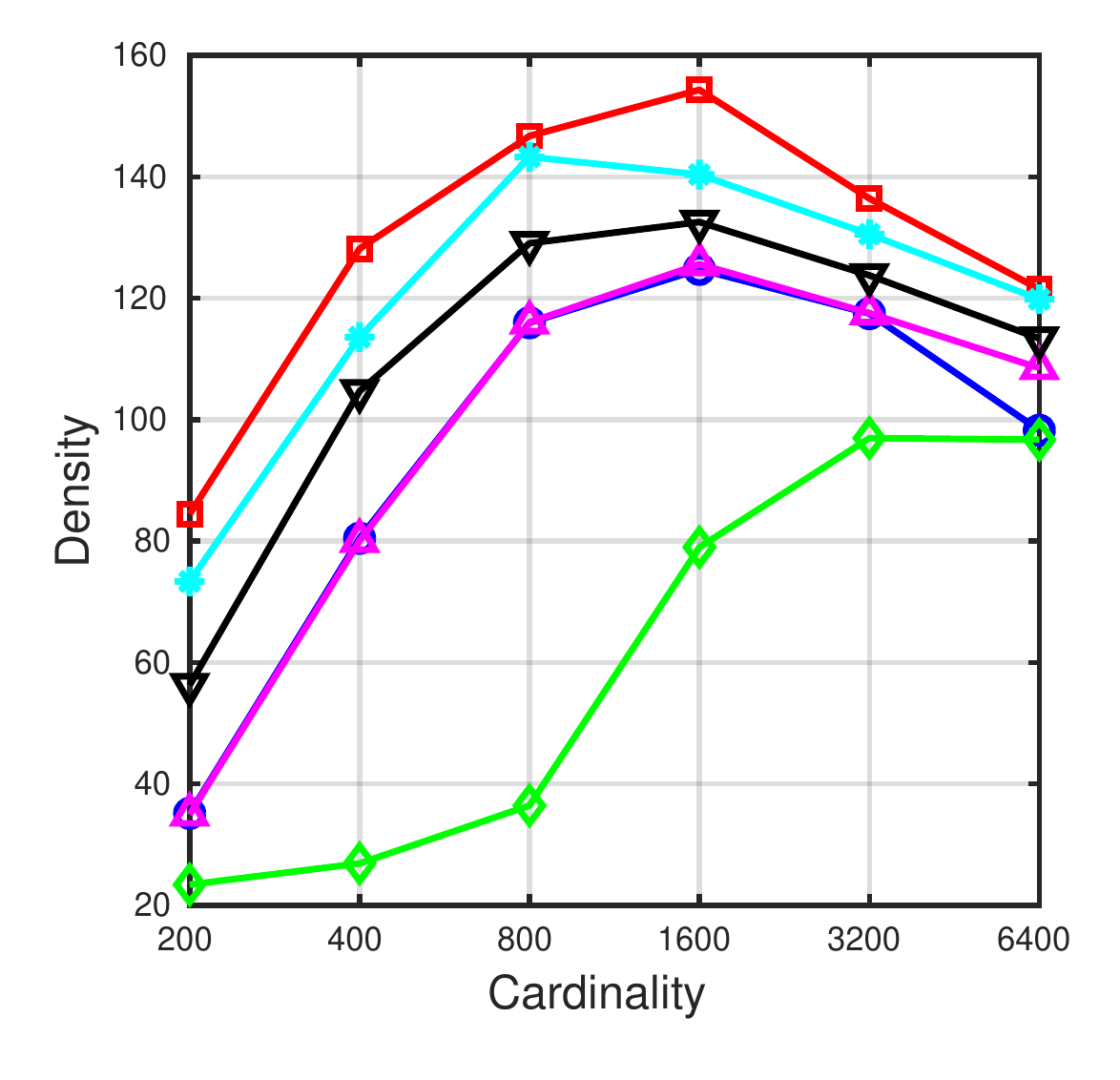}
\end{minipage}}
\addtocounter{subfigure}{-1}
\subfigure[~~(d)~ hollywood-2009]{
\begin{minipage}[htbp!]{0.195 \textwidth}
\centering
\vspace*{-0.cm}
\includegraphics[width=3.0cm]{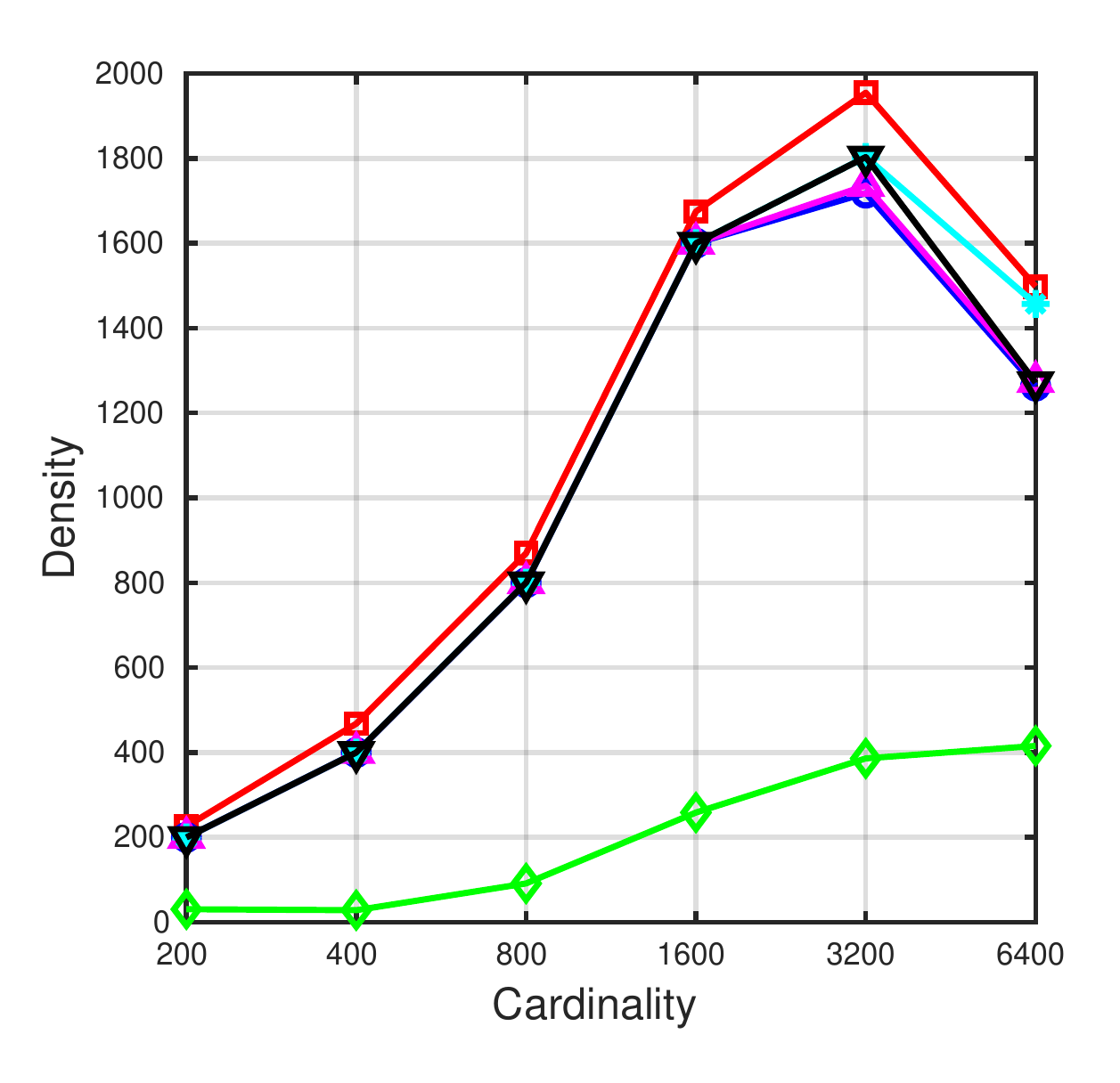}
\end{minipage}}
\vspace{-2ex}
\caption{Experimental results for dense subgraph discovery with $k\in \{ 200,400,800,1600,3200,6400 \}$. The metric is $\x^{\top} \W \x/k$.}
\label{fig:subgraph}
\vspace{-2ex}
\end{figure*}

\paragraph{Graph datasets.} The experiments for dense subgraph discovery are conducted on four large-scale graphs {\it uk-2007-05, dblp-2010, eswiki-2013}, and {\it hollywood-2009} from The Laboratory for Web Algorithmics\footnote{http://law.di.unimi.it/datasets.php}. Table \ref{table:graph} gives a brief description of each graph. For example,  hollywood-2009 contains roughly 1.13 million nodes and 113 million edges.

\begin{table}
\caption{The CPU time comparison (seconds) of dense subgraph discovery on the four graphs with $k=1600$.}
\center
\resizebox{0.75\linewidth}{!}{
\begin{tabular}{ccccc}
  \hline
 Method & uk-2007-05 & dblp-2010 & eswiki-2013 & hollywood-2009 \\
  \hline
  DPCD & \textbf{1.77} & \textbf{6.19}	& \textbf{14.17}	& \textbf{21.80}   \\
 LP & 5.02 & 7.27 & 58.96&115.38 \\
 RAVI &2.28&6.40& 104.94 &77.95\\
L2box-ADMM & 35.84 & 74.54 & 335.83&610.72 \\
 MPEC-EPM & 45.21 & 158.11 & 1454.07 &2114.20\\
 MPEC-ADM & 36.76 & 81.64 & 421.35 &720.66\\
\hline
\end{tabular}
}
\label{tab:cpu}
\end{table}

\paragraph{DPCD vs. state-of-the-art methods.}
The proposed DPCD is compared with the methods   LP \cite{hsieh2015pu}, RAVI \cite{ravi1994heuristic}, L2box-ADMM \cite{wu2018lp}, MPEC-EPM and MPEC-ADM \cite{yuan2016binary}, using the same objective function~\eqref{eq:Problem dense3}.
The cardinality $k$ is in the set $\{200,400,800,1600,3200,6400\}$. For the DPCD method, we run one $5$-neighborhood search ($m=5$ in the neighborhood search setting) after $10$ principal coordinate updates, and set the maximum iteration number for the principal coordinate update part to $100$. $L_1$ and $ L_2$ are updated by Eq. \eqref{eq:derive_L}. We tune the parameters $\alpha_1$ and $\alpha_2$ from $\{0.1,0.2,0.3,\ldots,0.9,1,2,3,\ldots,10\}$ by cross-validation according to datasets.  For other methods, we adopt the implementations and parameters suggested by the authors. The experimental results in Figure \ref{fig:subgraph} are reported in terms of $\x^{\top} \W \x/k$, which is the density of the subgraph with $k$ vertices. We can see that DPCD finds a denser subgraph than all compared methods in each case. Among the other methods, MPEC-EPM consistently outperforms LP in all the experiments.  RAVI generally leads to solutions with low density.
Furthermore, we provide CPU time comparisons for the six methods on the four graphs. From Table \ref{tab:cpu} we can see that, the proposed DPCD achieves the fastest runtime on all graphs. This is due to the fast updates in Algorithm \ref{Alg:1} (the main complexity is to calculate the gradient of the loss function, which can be done quickly). Also, the efficiency of our method becomes more obvious as the graph size increases. LP and RAVI are faster than the other three methods, since MPEC-EPM needs to run the LP procedure multiple times, and L2-box ADMM and MPEC-ADM usually need more iterations to converge.

\begin{table}
\caption{Comparison between DPCD with and without neighborhood search.}
\vspace{-1ex}
\center
\resizebox{0.85\linewidth}{!}{
\begin{tabular}{ccccc}
  \hline
 Method (Subgraph)   & uk-2007-05 & dblp-2010 & eswiki-2013 & hollywood-2009 \\
  \hline
  DPCD (loss function)  & \textbf{96.573} & \textbf{45.238}	& \textbf{158.340}	& \textbf{1944.921}   \\
 DPCD-0 (loss function)   & 93.5585 &  44.2512 & 147.815&1866.240 \\
 DPCD (run time)   &1.7689 & 6.1092	& 14.1745	  &21.8038   \\
 DPCD-0 (run time)   &  \textbf{1.1522} &   \textbf{5.5186} &   \textbf{9.9596} &  \textbf{12.8730} \\
\hline
\end{tabular}
}
\label{tab:neighbor}
\end{table}

\paragraph{With and without the neighborhood search.}
We conduct the comparison of the proposed method with its variant without neighborhood search technique, on the task of dense subgraph discovery with cardinality $1600$ and in terms of loss functions and run time.  DPCD-0 refers to our algorithm without neighborhood search.  From Table \ref{tab:neighbor} we can see that DPCD-0 is slightly faster, while DPCD usually achieves better results:


\subsection{Binary hashing}\label{sec:hashing}

\begin{table}
 \caption{Evaluation of DPCD and five general binary optimization methods with the same supervised loss function. The CIFAR-10 dataset is adopted. Results are reported in terms of MAP, Precision@500 and the training time.}
\vspace{-1ex}
\begin{center}
\resizebox{0.99\linewidth}{!}{
\begin{tabular}{cccccccccc}
\hline

\multirow{2}{*}{Method} &  \multicolumn{3}{c}{MAP} & \multicolumn{3}{c}{Precision@500} &  \multicolumn{3}{c}{Training time (seconds)}
\\

\cline{2-10}

 &  $32$ bits  &  $64$ bits  & $96$ bits  & $32$ bits  &
 $64$ bits  &  $96$ bits  & $32$ bits  & $64$ bits  &$96$ bits
\\
\hline
 DPCD   &   \textbf{0.7019}  &  \textbf{0.7088} &    \textbf{0.7126}  &  \textbf{0.6337} &  \textbf{0.6353} & \textbf{0.6370}    &  \textbf{3.76} &  \textbf{6.01} & \textbf{9.39}
\\
 DCC     &  0.5941  &  0.6193 & 0.6314  &  0.5486 &  0.5766 & 0.5894   &   11.18&   36.03&  158.87
\\
 SGM   & 0.6856  &  0.6986 & 0.7013  &  0.6177&  0.6308  & 0.6360     &   7.87 &   10.83 &   16.30
\\
 LP   &  0.5237  &  0.5468  &  0.5459   &    0.4704  &  0.4972  &  0.4866 & {4.52}&    7.96&   13.64
\\
 L2box-ADMM   & 0.6399  &  0.6724 & 0.6830  &  0.5929&  0.6095  & 0.6162    &   43.07 & 90.10   & 200.94
\\
 MPEC-EPM   &  0.5823  &  0.6253  &      0.6276   &    0.5385&   0.5738 & 0.5790 & 36.36&  124.54&  260.22
\\
\hline
\end{tabular}\label{tab:UB1}
}
\end{center}
\end{table}


Binary hashing aims to encode high-dimensional data points, such as images and videos, into compact binary hash codes such that the similarities between the original data points and hash codes are preserved. This can be used to provide a constant or sub-linear search time and reduce the storage cost dramatically for such data points. The efficiency and effectiveness of binary hashing make it a popular technique in machine learning, information retrieval and computer vision \cite{gui2018fast,hu2018hashing,liu2014discrete,shen2018unsupervised,wang2018survey}.
In a typical binary hashing task,  $\X=(\x_1,\x_2,\ldots,\x_n)^\top  \in \mathbb{R}^{n\times d}$ denotes a matrix of the original data points, where $\x_i \in \mathbb{R}^d$ is the $i$-th sample point, $n$ is the number of samples, and $d$ is the dimension of each sample. In the supervised setting, we let $\Y \in \mathbb{R}^{n\times c}$ be the label matrix, i.e., $\Y_{i,j} = 1$ if $\x_i$ belongs to the $j$-th class and $0$ otherwise. The aim is to map $\X$ to some $\B=(\b_1,\b_2,\ldots,\b_n)^\top \in \{\pm 1\}^{n\times r} $, i.e., map each $x_i$ to a binary code $\b_i \in \{\pm 1\}^r $ for some small integer~$r$ and preserve some similarities between the original data points in $\X$ and hash codes in $\B$.

\paragraph{Image datasets.} Three large-scale image datasets, CIFAR-10\footnote{http://www.cs.toronto.edu/kriz/cifar.html.}, ImageNet\footnote{http://www.image-net.org/.}, 
and NUS-WIDE\footnote{http://lms.comp.nus.edu.sg/research/NUS-WIDE.htm.}, are used in the binary hashing experiments. CIFAR-10 has 60k images, which are divided into 10 classes with 6k images each. We use a 384-dimensional GIST feature vector \cite{oliva2001modeling} to represent each image. 59k images are selected as the training set and the test set contains the remaining 1k images.
A subset of the ImageNet, ILSVRC 2012, contains about 1.2 million images with 1k categories. As in \cite{hu2018hashing,shen2015supervised}, we use 4096-dimensional deep feature vectors for each image, take 127K training images from the 100 largest classes, and 50K images from the validation set as the test set. 
NUS-WIDE contains about 270K images with 81 labels. The images may have multiple labels. The 500-dimensional Bag-of-Words features are used here \cite{nus-wide-civr09}. We adopt the 21 most frequent labels with the corresponding 193K images. For each label, 100 images are randomly selected as the test set and the remaining as the training set.

\paragraph{DPCD vs. general binary optimization methods.} To illustrate the efficiency and effectiveness of our algorithm, we compare DPCD with several general state-of-the-art binary optimization methods DCC  \cite{shen2015supervised}, SGM \cite{liu2014discrete}, LP \cite{hsieh2015pu}, L2box-ADMM \cite{wu2018lp}, and MPEC-EPM \cite{yuan2016binary}, on the dataset CIFAR-10. Various loss functions are designed for binary hashing (for examples, see Tables \ref{tab:UB2} and \ref{tab:UB3}). For fair comparison, we adopt the widely used supervised objective function  \cite{gui2018fast,shen2015supervised}
\begin{align}
    f(\B, \W)= \frac12\|\Y-\B\W\|_2^2 + \frac{\delta}2 \|\W\|_2^2
\end{align}
for each method.
Thus the optimization problem becomes:
\begin{align}\label{eq:sdh_loss}
\min_{\mathbf{B},\W}~  \frac12\|\Y-\B\W\|_2^2 + \frac{\delta}2 \|\W\|_2^2 ~~~~~~
  \St ~~~ \mathbf{B}\in   \{ \pm 1\}^{n\times r},
  \W \in \mathbb{R}^{r\times c},
\end{align}
where $\beta$ is a regularization parameter,  and $\W \in \mathbb{R}^{r\times c}$ is the projection matrix (see \cite{shen2015supervised}) which will be learned jointly with $\B$. The whole optimization runs iteratively over $\B$ and $\W$. When $\W$ is fixed, we apply DPCD algorithm to $\B$. The key step is to calculate the gradient of $f(\B, \W)$ as:
\begin{align}\label{eq:gradientB}
    \nabla_{\B}f(\B, \W) = (\B\W-\Y)\W^\top.
\end{align}
Then $L_1, L_2$ can be obtained by Eq. \eqref{eq:derive_L}. After deriving $ S^{k+} $  and $ S^{k-}$, we update $\B$ by Eq. \eqref{eq:updateBk}. When $\B$ is fixed, $\W$ can be updated by
\begin{align}\label{eq: update W}
\W={\arg}\min_{\W^*\in \mathbb{R}^{r\times c}}f(\B, \W^*)=(\B^\top \B+\beta \I_r)^{-1} \B^\top \Y.
\end{align}
Finally, we adopt the linear hash function $h(\X) =\sgn(\X\P)$ to encode $\X$ onto binary codes, where $\P \in \mathbb{R}^{d\times r}$ can be derived by:
\begin{align}\label{linear regression1}
\P ={\arg}\min_{\P^*\in \mathbb{R}^{d\times r}}   \|\X\P^*-\B\|^2=(\X^\top\X)^{-1}\X^\top \B.
  \hspace{12pt}
\end{align}
During the test phase, for a query item, first we use the above linear hash function to derive the hash code, then adopt nearest neighbor search under Hamming distance to find its similar items.
For the DPCD method, we run one $5$-neighborhood search after $10$ principal coordinate updates, and tune the parameters $\alpha_1$ and $\alpha_2$ from $\{0.1,0.2,0.3,\ldots,0.9,1,2,3,\ldots,10\}$ by cross-validation according to datasets and binary code lengths, and set the maximum iteration number for $\B$ to $20$ each time when $\W$ is fixed. We run at most five iterations for updating $\B$ and $\W$ iteratively.  For other methods, we adopt the implementations and parameters suggested by the authors. Ground truths are defined by the label information from the datasets. The experimental results are reported in terms of mean average precision (MAP), Precision@500 (Precision@500 refers to the ratio of the number of retrieved true positive items among 500 nearest neighbors to 500.) and training time efficiency (we ignore the comparison of test time here since the test parts are similar for each algorithm). The experiments are conducted on the   CIFAR-10 dataset. From Table \ref{tab:UB1} we can see that the proposed DPCD outperforms all other methods in MAP and Precision@500. For instance, on CIFAR-10 with 96 bits, DPCD outperforms DCC by 8.1$\%$, and MPEC-EPM by 8.5$\%$ in terms of MAP. It is clear that increasing the number of bits yields better performance for all methods.  Also, the training time for the proposed DPCD method is always faster than other compared methods. For example, DPCD runs 20 times faster than MPEC-EPM on $64$ bits, which verifies that one major advantage of our method is the fast optimization process.

\paragraph{Algorithm complexity analysis.} Now we discuss the complexity of the above supervised DPCD algorithm.
The calculation of the gradient $\nabla_{\B}f(\B^k, \W)$ is obtained by Eq.  \eqref{eq:gradientB}, thus the complexity is  $O(nrc)$. Then $L_1$ and $L_2$ are derived by Eq. \eqref{eq:derive_L}, which has complexity $O(nr)$ since there are  $nr$ additions in Eq. \eqref{eq:derive_L}. Furthermore,  $ S^{k+} $  and $ S^{k-}$  can be determined by running through all $\nabla_{ij} f(\B^k)$ where $1\leq i\leq n, \, 1\leq j\leq r$, which also has complexity $O(nr)$. Similarly, the update for~$\B$ has complexity $O(nr)$. Let $T$ be the maximum
iteration number during the $\B$ updating step (the number of principal coordinate update parts of DPCD). Then, the total time complexity of updating $\B$ is $O(Tnrc)$.
When $\B$ is given, the complexity for updating $\W$ in Eq. \eqref{eq: update W} is $O(ncr+r^3+cr^2)$.  The complexity for calculating the matrix $\P$ in Eq. \eqref{linear regression1} is $O(d^3+nd^2+ndr)$.
Suppose that there are at most $t$ iterations for updating $\B$ and $\W$ iteratively.
Since $r,c \ll n$, the total complexity for supervised DPCD is $O\left(tTncr  +  d^3+nd^2+ndr \right)$, which is a linear function of $n$.

\paragraph{DPCD vs. specific binary hashing methods.} The proposed method is not only fast, but also very versatile, and can thus be used to handle many different loss functions. To demonstrate this, we apply DPCD to several widely used loss functions, and compare the results with the corresponding hashing methods that were designed for each specific loss function. The optimization process is similar to the supervised case.  Table \ref{tab:UB2} illustrates a comparison of DPCD with three unsupervised methods,  SADH-L \cite{shen2018unsupervised}, ARE \cite{hu2018hashing}, and ITQ \cite{gong2013iterative}, on the ImageNet dataset, while Table \ref{tab:UB3} shows a comparison with supervised hashing methods, SDH  \cite{shen2015supervised}, FSDH \cite{gui2018fast}, and FastHash \cite{lin2014fast}, on NUS-WIDE, using their specific loss functions. The proposed DPCD shows increased performance over the original methods and achieves higher MAP and Precision@500 in most cases, especially for the supervised loss functions. In terms of the training time, DPCD outperforms all methods except FSDH. The proposed method can significantly decrease the training time for unsupervised loss functions due to the fast updating of the binary codes $\B$.  Finally, we conclude that DPCD is a fast and effective optimization method for large-scale image retrieval tasks.

\begin{table}[t!]
 \caption{Evaluation of the proposed DPCD method and three unsupervised methods. The ImageNet dataset is adopted. Results are reported in terms of MAP, Precision@500 and training time.}
\begin{center}
\resizebox{0.99\linewidth}{!}{
\begin{tabular}{ccccccccccc}
\hline

\multirow{2}{*}{Method} & \multirow{2}{*}{Loss Function} &  \multicolumn{3}{c}{MAP} & \multicolumn{3}{c}{Precision@500} &  \multicolumn{3}{c}{Training time (s)}
\\

\cline{3-11}

& &  $32$ bits  &  $64$ bits  & $96$ bits  & $32$ bits  &
 $64$ bits  &  $96$ bits  & $32$ bits  & $64$ bits  &$96$ bits
\\
\hline
 SADH-L  &  \multirow{2}{*}{$\min_{\mathbf{B}}\mathbf{Tr}(\B^\top \mathbf{L}\B)$}  &     0.2448  &  \textbf{0.3104} &   0.3294  &  0.3692  &  \textbf{0.4588}   & 0.4835    &  121.90&  384.01&  738.33
\\
 DPCD   &  &  \textbf{0.2612}  &  0.3085  &  \textbf{0.3441}  &  \textbf{0.3945}  &  0.4508  &   \textbf{0.4992}  &   \textbf{19.03}&   \textbf{33.50}&   \textbf{48.29}
\\
 \hline
 ARE  &  \multirow{2}{*}{$\min_{\mathbf{B}}\|\B\B^\top-r\X\X^\top\|^2$}  &     0.2509 &   0.2997  &  0.3276  &  0.3626 &   0.4478   & 0.4724    &  244.46&  287.19&  332.95
\\
 DPCD   &  &  \textbf{0.2808}  &  \textbf{0.3316}  &  \textbf{0.3607}  &  \textbf{0.3970}  & \textbf{ 0.4856 } &   \textbf{0.5112}  &   \textbf{6.59}&   \textbf{8.37}&   \textbf{11.24}
\\
\hline
 ITQ  &  \multirow{2}{*}{$\min_{\mathbf{B},\R}\|\B-\X\W\R\|^2$}  &     \textbf{0.3209}  &  0.4075  &  0.4388  &  \textbf{0.4269}  &  0.5208  &  0.5581    &  26.01&   32.97&   33.39
\\
 DPCD   &  &  0.3093  &  \textbf{0.4155}  &  \textbf{0.4467} &  0.4148  &  \textbf{0.5283}   & \textbf{0.5654}  &   \textbf{3.35}&   \textbf{7.23}&   \textbf{8.02}
\\
\hline
\end{tabular}\label{tab:UB2}
}
\end{center}
\end{table}

\begin{table}[t!]
 \caption{Evaluation of the proposed DPCD method and three supervised methods. The NUS-WIDE dataset is adopted. Results are reported in terms of MAP, Precision@500 and training time.}
\begin{center}
\resizebox{0.99\linewidth}{!}{
\begin{tabular}{ccccccccccc}
\hline

\multirow{2}{*}{Method} & \multirow{2}{*}{Loss Function} &  \multicolumn{3}{c}{MAP} & \multicolumn{3}{c}{Precision@500} &  \multicolumn{3}{c}{Training time (s)}
\\

\cline{3-11}

& &  $32$ bits  &  $64$ bits  & $96$ bits  & $32$ bits  &
 $64$ bits  &  $96$ bits  & $32$ bits  & $64$ bits  &$96$ bits
\\
\hline
 SDH  &  \multirow{2}{*}{$\min_{\mathbf{B},\W}\|\Y-\B\W\|^2 +{\delta} \|\W\|^2$}  &     0.5716  &   {0.5827} &   0.5920  &  0.6041  &   {0.6056}   & 0.6189    &  24.53&  107.50&  486.95
\\
 DPCD   &  &  \textbf{0.6124}  &  \textbf{0.6230}  &  \textbf{0.6392}  &  \textbf{0.6287}  &  \textbf{0.6357}  &   \textbf{0.6502}  &  \textbf{ 6.91}&  \textbf{16.39}&   \textbf{22.53}
\\
 \hline
 FSDH  &  \multirow{2}{*}{$\min_{\mathbf{B},\W}\|\B-\Y\W\|_2^2 +{\delta} \|\W\|_2^2$}  &     0.5690 &   0.5639  &  0.5676  &  0.5918 &   0.5860   & 0.5988    &  \textbf{2.76}&   \textbf{3.07}&   \textbf{5.95}
\\
 DPCD   &  &  \textbf{0.6159}  &  \textbf{0.6170}  &  \textbf{0.6253}  &  \textbf{0.6260}  & \textbf{ 0.6297 } &   \textbf{0.6335}  &   5.42&   8.96&   10.99
\\
\hline
 FastHash  &  \multirow{2}{*}{$\min_{\mathbf{B}}\|\B\B^\top-r\Y\|^2$}  &     {0.5174}  &  0.5398  &  0.5403  &  {0.5867}  &  0.6014  &  0.6180    &  1381.75&   4668.13&  10605.82
\\
 DPCD   &  &  \textbf{0.5621}  &  \textbf{0.5579}  &  \textbf{0.5767} &  \textbf{0.5991}  &  \textbf{0.6145}   & \textbf{0.6226}  &   \textbf{6.42}&    \textbf{7.88}&   \textbf{11.79}
\\
\hline
\end{tabular}\label{tab:UB3}
}
\end{center}
\end{table}

\section{Conclusion and future work}
This paper presents a novel fast optimization method, called Discrete Principal Coordinate Descent (DPCD), to approximately solve binary optimization problems with/without restrictions on the numbers of $1$s and $-1$s in the variables. We derive several theoretical results on the convergence of the proposed algorithm.  Experiments on dense subgraph discovery and binary hashing demonstrate that our method generally outperforms state-of-the-art methods in terms of both solution quality and optimization efficiency.

In the future, we plan to extend our algorithm to a more general framework. Our methods can be seen as a discrete version of the normal gradient descent methods. Since the gradient descent has several useful variants such as momentum, Adam and Adagrad methods \cite{qian1999momentum,ruder2016overview}, it would be possible to combine our methods with these gradient-based methods and propose some discrete versions of them.
We would like to explore more on this direction in the future.

\bibliographystyle{plain}

\begin{thebibliography}{10}

\bibitem{alvarez2006large}
J~Ignacio Alvarez-Hamelin, Luca Dall'Asta, Alain Barrat, and Alessandro
  Vespignani.
\newblock Large scale networks fingerprinting and visualization using the
  k-core decomposition.
\newblock In {\em Advances in Neural Information Processing Systems (NIPS)},
  pages 41--50, 2006.

\bibitem{ames2015guaranteed}
Brendan~PW Ames.
\newblock Guaranteed recovery of planted cliques and dense subgraphs by convex
  relaxation.
\newblock {\em Journal of Optimization Theory and Applications},
  167(2):653--675, 2015.

\bibitem{angel2012dense}
Albert Angel, Nikos Sarkas, Nick Koudas, and Divesh Srivastava.
\newblock Dense subgraph maintenance under streaming edge weight updates for
  real-time story identification.
\newblock {\em Proceedings of the VLDB Endowment}, 5(6):574--585, 2012.

\bibitem{balalau2015finding}
Oana~Denisa Balalau, Francesco Bonchi, TH~Chan, Francesco Gullo, and Mauro
  Sozio.
\newblock Finding subgraphs with maximum total density and limited overlap.
\newblock In {\em Proceedings of the Eighth ACM International Conference on Web
  Search and Data Mining}, pages 379--388. ACM, 2015.

\bibitem{bi2014exact}
Shujun Bi, Xiaolan Liu, and Shaohua Pan.
\newblock Exact penalty decomposition method for zero-norm minimization based
  on mpec formulation.
\newblock {\em SIAM Journal on Scientific Computing}, 36(4):A1451--A1477, 2014.

\bibitem{borwein2010convex}
Jonathan Borwein and Adrian~S Lewis.
\newblock {\em Convex analysis and nonlinear optimization: theory and
  examples}.
\newblock Springer Science \& Business Media, 2010.

\bibitem{boyd2011distributed}
Stephen Boyd, Neal Parikh, Eric Chu, Borja Peleato, Jonathan Eckstein, et~al.
\newblock Distributed optimization and statistical learning via the alternating
  direction method of multipliers.
\newblock {\em Foundations and Trends{\textregistered} in Machine learning},
  3(1):1--122, 2011.

\bibitem{boyd2004convex}
Stephen Boyd and Lieven Vandenberghe.
\newblock {\em Convex optimization}.
\newblock Cambridge university press, 2004.

\bibitem{nus-wide-civr09}
Tat-Seng Chua, Jinhui Tang, Richang Hong, Haojie Li, Zhiping Luo, and Yan-Tao
  Zheng.
\newblock Nus-wide: A real-world web image database from national university of
  singapore.
\newblock In {\em Proc. of ACM Conf. on Image and Video Retrieval (CIVR'09)},
  Santorini, Greece., July 8-10, 2009.

\bibitem{gong2013iterative}
Yunchao Gong, Svetlana Lazebnik, Albert Gordo, and Florent Perronnin.
\newblock Iterative quantization: A procrustean approach to learning binary
  codes for large-scale image retrieval.
\newblock {\em IEEE Transactions on Pattern Analysis and Machine Intelligence},
  35(12):2916--2929, 2013.

\bibitem{gui2018fast}
Jie Gui, Tongliang Liu, Zhenan Sun, Dacheng Tao, and Tieniu Tan.
\newblock Fast supervised discrete hashing.
\newblock {\em IEEE Transactions on Pattern Analysis and Machine Intelligence},
  40(2):490--496, 2018.

\bibitem{he2016joint}
Lifang He, Chun-Ta Lu, Jiaqi Ma, Jianping Cao, Linlin Shen, and Philip~S Yu.
\newblock Joint community and structural hole spanner detection via harmonic
  modularity.
\newblock In {\em Proceedings of the 22nd ACM SIGKDD International Conference
  on Knowledge Discovery and Data Mining}, pages 875--884. ACM, 2016.

\bibitem{hsieh2015pu}
Cho-Jui Hsieh, Nagarajan Natarajan, and Inderjit~S Dhillon.
\newblock Pu learning for matrix completion.
\newblock In {\em ICML}, pages 2445--2453, 2015.

\bibitem{hu2018hashing}
Mengqiu Hu, Yang Yang, Fumin Shen, Ning Xie, and Heng~Tao Shen.
\newblock Hashing with angular reconstructive embeddings.
\newblock {\em IEEE Transactions on Image Processing}, 27(2):545--555, 2018.

\bibitem{johnson1979computers}
David~S Johnson and Michael~R Garey.
\newblock {\em Computers and intractability: A guide to the theory of
  NP-completeness}, volume~1.
\newblock WH Freeman San Francisco, 1979.

\bibitem{komodakis2007approximate}
Nikos Komodakis and Georgios Tziritas.
\newblock Approximate labeling via graph cuts based on linear programming.
\newblock {\em IEEE Transactions on Pattern Analysis and Machine Intelligence},
  29(8):1436--1453, 2007.

\bibitem{li2015global}
Guoyin Li and Ting~Kei Pong.
\newblock Global convergence of splitting methods for nonconvex composite
  optimization.
\newblock {\em SIAM Journal on Optimization}, 25(4):2434--2460, 2015.

\bibitem{lin2014fast}
Guosheng Lin, Chunhua Shen, Qinfeng Shi, Anton Van~den Hengel, and David Suter.
\newblock Fast supervised hashing with decision trees for high-dimensional
  data.
\newblock In {\em Proceedings of the IEEE Conference on Computer Vision and
  Pattern Recognition}, pages 1963--1970, 2014.

\bibitem{lin2013general}
Guosheng Lin, Chunhua Shen, David Suter, and Anton Van Den~Hengel.
\newblock A general two-step approach to learning-based hashing.
\newblock In {\em Proceedings of the IEEE international conference on computer
  vision}, pages 2552--2559, 2013.

\bibitem{liu2017discretely}
Li~Liu, Ling Shao, Fumin Shen, and Mengyang Yu.
\newblock Discretely coding semantic rank orders for supervised image hashing.
\newblock In {\em Proceedings of the IEEE Conference on Computer Vision and
  Pattern Recognition}, pages 1425--1434, 2017.

\bibitem{liu2014discrete}
Wei Liu, Cun Mu, Sanjiv Kumar, and Shih-Fu Chang.
\newblock Discrete graph hashing.
\newblock In {\em Advances in Neural Information Processing Systems (NIPS)},
  pages 3419--3427, 2014.

\bibitem{lu2013sparse}
Zhaosong Lu and Yong Zhang.
\newblock Sparse approximation via penalty decomposition methods.
\newblock {\em SIAM Journal on Optimization}, 23(4):2448--2478, 2013.

\bibitem{luo2018robust}
Yadan Luo, Yang Yang, Fumin Shen, Zi~Huang, Pan Zhou, and Heng~Tao Shen.
\newblock Robust discrete code modeling for supervised hashing.
\newblock {\em Pattern Recognition}, 75:128--135, 2018.

\bibitem{mehrotra1992implementation}
Sanjay Mehrotra.
\newblock On the implementation of a primal-dual interior point method.
\newblock {\em SIAM Journal on Optimization}, 2(4):575--601, 1992.

\bibitem{murray2010algorithm}
Walter Murray and Kien-Ming Ng.
\newblock An algorithm for nonlinear optimization problems with binary
  variables.
\newblock {\em Computational Optimization and Applications}, 47(2):257--288,
  2010.

\bibitem{oliva2001modeling}
Aude Oliva and Antonio Torralba.
\newblock Modeling the shape of the scene: A holistic representation of the
  spatial envelope.
\newblock {\em International Journal of Computer Vision}, 42(3):145--175, 2001.

\bibitem{olsson2007solving}
Carl Olsson, Anders~P Eriksson, and Fredrik Kahl.
\newblock Solving large scale binary quadratic problems: Spectral methods vs.
  semidefinite programming.
\newblock In {\em 2007 IEEE Conference on Computer Vision and Pattern
  Recognition}, pages 1--8. IEEE, 2007.

\bibitem{qian1999momentum}
Ning Qian.
\newblock On the momentum term in gradient descent learning algorithms.
\newblock {\em Neural networks}, 12(1):145--151, 1999.

\bibitem{ravi1994heuristic}
Sekharipuram~S Ravi, Daniel~J Rosenkrantz, and Giri~Kumar Tayi.
\newblock Heuristic and special case algorithms for dispersion problems.
\newblock {\em Operations Research}, 42(2):299--310, 1994.

\bibitem{rozenshtein2018finding}
Polina Rozenshtein, Francesco Bonchi, Aristides Gionis, Mauro Sozio, and
  Nikolaj Tatti.
\newblock Finding events in temporal networks: Segmentation meets
  densest-subgraph discovery.
\newblock In {\em 2018 IEEE International Conference on Data Mining (ICDM)},
  pages 397--406. IEEE, 2018.

\bibitem{ruder2016overview}
Sebastian Ruder.
\newblock An overview of gradient descent optimization algorithms.
\newblock {\em arXiv preprint arXiv:1609.04747}, 2016.

\bibitem{shen2015supervised}
Fumin Shen, Chunhua Shen, Wei Liu, and Heng Tao~Shen.
\newblock Supervised discrete hashing.
\newblock In {\em Proceedings of the IEEE Conference on Computer Vision and
  Pattern Recognition}, pages 37--45, 2015.

\bibitem{shen2018unsupervised}
Fumin Shen, Yan Xu, Li~Liu, Yang Yang, Zi~Huang, and Heng~Tao Shen.
\newblock Unsupervised deep hashing with similarity-adaptive and discrete
  optimization.
\newblock {\em IEEE Transactions on Pattern Analysis and Machine Intelligence},
  2018.

\bibitem{shi2013multi}
Xinchu Shi, Haibin Ling, Junling Xing, and Weiming Hu.
\newblock Multi-target tracking by rank-1 tensor approximation.
\newblock In {\em Proceedings of the IEEE Conference on Computer Vision and
  Pattern Recognition}, pages 2387--2394, 2013.

\bibitem{wang2018survey}
Jingdong Wang, Ting Zhang, Nicu Sebe, Heng~Tao Shen, et~al.
\newblock A survey on learning to hash.
\newblock {\em IEEE Transactions on Pattern Analysis and Machine Intelligence},
  40(4):769--790, 2018.

\bibitem{wang2011information}
Meihong Wang and Fei Sha.
\newblock Information theoretical clustering via semidefinite programming.
\newblock In {\em Proceedings of the Fourteenth International Conference on
  Artificial Intelligence and Statistics}, pages 761--769, 2011.

\bibitem{wang2017large}
Peng Wang, Chunhua Shen, Anton van~den Hengel, and Philip~HS Torr.
\newblock Large-scale binary quadratic optimization using semidefinite
  relaxation and applications.
\newblock {\em IEEE Transactions on Pattern Analysis and Machine Intelligence},
  39(3):470--485, 2017.

\bibitem{wang2015global}
Yu~Wang, Wotao Yin, and Jinshan Zeng.
\newblock Global convergence of admm in nonconvex nonsmooth optimization.
\newblock {\em Journal of Scientific Computing}, pages 1--35, 2015.

\bibitem{weiss2009spectral}
Yair Weiss, Antonio Torralba, and Rob Fergus.
\newblock Spectral hashing.
\newblock In {\em Advances in Neural Information Processing Systems (NIPS)},
  pages 1753--1760, 2009.

\bibitem{wright2015coordinate}
Stephen~J Wright.
\newblock Coordinate descent algorithms.
\newblock {\em Mathematical Programming}, 151(1):3--34, 2015.

\bibitem{wu2018lp}
Baoyuan Wu and Bernard Ghanem.
\newblock lp-box admm: A versatile framework for integer programming.
\newblock {\em IEEE Transactions on Pattern Analysis and Machine Intelligence},
  2018.

\bibitem{xiong2021generalized}
Huan Xiong, Mengyang Yu, Li~Liu, Fan Zhu, Jie Qin, Fumin Shen, and Ling Shao.
\newblock A generalized method for binary optimization: Convergence analysis
  and applications.
\newblock {\em IEEE Transactions on Pattern Analysis and Machine Intelligence},
  2021.

\bibitem{yuan2016binary}
Ganzhao Yuan and Bernard Ghanem.
\newblock Binary optimization via mathematical programming with equilibrium
  constraints.
\newblock {\em arXiv preprint arXiv:1608.04425}, 2016.

\bibitem{yuan2016sparsity}
Ganzhao Yuan and Bernard Ghanem.
\newblock Sparsity constrained minimization via mathematical programming with
  equilibrium constraints.
\newblock {\em arXiv preprint arXiv:1608.04430}, 2016.

\bibitem{yuan2017exact}
Ganzhao Yuan and Bernard Ghanem.
\newblock An exact penalty method for binary optimization based on mpec
  formulation.
\newblock In {\em AAAI}, pages 2867--2875, 2017.

\bibitem{yuan2013truncated}
Xiao-Tong Yuan and Tong Zhang.
\newblock Truncated power method for sparse eigenvalue problems.
\newblock {\em Journal of Machine Learning Research}, 14(Apr):899--925, 2013.

\bibitem{zhang2005general}
Bin Zhang and Steve Horvath.
\newblock A general framework for weighted gene co-expression network analysis.
\newblock {\em Statistical applications in genetics and molecular biology},
  4(1), 2005.

\bibitem{zhang2007binary}
Zhongyuan Zhang, Tao Li, Chris Ding, and Xiangsun Zhang.
\newblock Binary matrix factorization with applications.
\newblock In {\em Data Mining, 2007. ICDM 2007. Seventh IEEE International
  Conference on}, pages 391--400. IEEE, 2007.

\end{thebibliography}

\end{document}